\documentclass[11pt]{amsart}
\usepackage[titletoc]{appendix}
\usepackage[foot]{amsaddr}

\calclayout

\usepackage{amssymb,amsmath,amsthm}
\usepackage{enumerate}
\usepackage{mathtools}

\usepackage[utf8]{inputenc}

\usepackage{color}
\usepackage[colorlinks,linkcolor=black,citecolor=black]{hyperref}

\usepackage{tikz-cd}
\tikzcdset{
arrow style=tikz,
diagrams={>={Straight Barb[scale=0.8]}}
}

\def\AA{\mathord{\mathbf{A}}}
\def\CC{\mathord{\mathbf{C}}}
\def\RR{\mathord{\mathbf{R}}}
\def\QQ{\mathord{\mathbf{Q}}}

\def\ZZ{\mathord{\mathbf{Z}}}

\def\NN{\mathord{\mathbf{N}}}

\def\mm{\mathord{\mathfrak{m}}}

\def\ord{{\rm ord}}
\def\pr{{\rm pr}}

\def\End{\mathop{\rm End}}

\def\Im{\mathop{\rm Im}}
\def\Re{\mathop{\rm Re}}

\def\Spec{\mathop{\rm Spec}}

\def\SL{\mathop{\rm SL}}

\def\trdeg{\text{\rm trdeg}}

\def\dR{\text{\rm dR}}

\def\mult{\text{\rm mult}}

\def\comp{\text{\rm comp}}

\def\defeq{\coloneqq}

\def\tensor{\otimes}

\def\to{\longrightarrow}
\def\mapsto{\longmapsto}

\newtheorem{theorem}{Theorem}[section]
\newtheorem{prop}[theorem]{Proposition}
\newtheorem{conj}[theorem]{Conjecture}
\newtheorem{lemma}[theorem]{Lemma}

\newtheorem{coro}[theorem]{Corollary}

\theoremstyle{definition}
\newtheorem{ex}[theorem]{Example}
\newtheorem{defi}[theorem]{Definition}

\newtheorem{obs}[theorem]{Remark}

\newtheorem{exo}[theorem]{Exercise}

\numberwithin{equation}{section}

\title{A geometric introduction to transcendence questions on values of modular forms}
\author{Tiago J. Fonseca}

\address{Mathematical Institute, University of Oxford, Andrew Wiles Building, Radcliffe Observatory Quarter, Woodstock Road, Oxford, OX2 6GG, United Kingdon}
\email{tiago.jardimdafonseca@maths.ox.ac.uk}

\begin{document}
\maketitle

\tableofcontents

\section{Introduction}

One of the most striking arithmetical applications of Ramanujan's relations between the normalised Eisenstein series $E_2$, $E_4$, $E_6$ (see \cite{ramanujan} and \cite{NP01} Chapter 1), namely
\begin{align} \label{rameq} \tag{R}
\frac{1}{2\pi i}\frac{dE_2}{d\tau} = \frac{E_2^2 - E_4}{12}\text{, }\ \ \frac{1}{2\pi i}\frac{dE_4}{d\tau} = \frac{E_2E_4 - E_6}{3}\text{, } \ \ \frac{1}{2\pi i}\frac{dE_6}{d\tau} = \frac{E_2E_6 - E_4^2}{2}
\text{,}
\end{align}
is the following algebraic independence theorem proved by Nesterenko in 1996. 

\begin{theorem}[\cite{nesterenko96}]
 For every $\tau \in \mathbf{H} = \{z \in \CC \mid \Im z >0\}$, we have
 $$
 \trdeg_{\QQ}\QQ(e^{2\pi i \tau},E_2(\tau),E_4(\tau),E_6(\tau))\ge 3\text{.}
 $$
\end{theorem}

This means that among the four complex numbers $e^{2\pi i \tau}$, $E_2(\tau)$, $E_4(\tau)$, and $E_6(\tau)$ there are always three of them which are algebraically independent over $\QQ$.

Nesterenko's result is remarkable both in its short and powerful statement as in its proof method. In the next sections, we explain some applications of Nesterenko's theorem and the main ideas of its proof. For other accounts of Nesterenko's proof, the reader may consult, besides the original paper \cite{nesterenko96}, the collective volumes \cite{NP01} and \cite{FGK05}. Here, we shall emphasise the special role played by the dynamics of the algebraic differential equations (\ref{rameq}) in the guise of Nesterenko's `$D$-property'. 

To help the reader with no background in Transcendental Number Theory, we shall start with a brief overview of some of its main concepts and results. Let us point out, however, that some major results such as Baker's theorems or Wüstholz analytic subgroup theorem are not discussed. For more complete and better written introductions to this same subject, we refer to the classic \cite{baker75}, or to the more recent \cite{MR14}.

In the last section, we give a very short introduction to `periods' and discuss some general transcendence questions related to Nesterenko's theorem. Periods are complex numbers given by integrals in algebraic geometry which have recently gained much attention of the Number Theory community due to their deep connections with the theory of motives. This places Nesterenko's theorem in a larger context and allows us not only to better appreciate its content, but also to dream and speculate on future generalisations.

\section{A biased overview of transcendence theory}

Transcendence theory is one of the oldest, and reputedly one of the most difficult, domains of Mathematics. Here, we can only scratch the surface.

\subsection{First notions}

In mathematics, `transcendental' is the antonym of `algebraic'. Accordingly, a complex number $\alpha$ is said to be \emph{transcendental} if it is not algebraic --- that is,  $P(\alpha)\neq 0$ for every $P \in \QQ[X]\setminus\{0\}$. Similarly, we say that a function $f$ (seen either as a formal Laurent series in $\CC(\!(t)\!)$ or as a meromorphic function on some open domain of $\CC$) is transcendental if it is not an algebraic function: $P(t,f(t))\not\equiv 0$ for every $P \in \CC[X,Y]\setminus\{0\}$.

The following definition generalises these notions.

\begin{defi}
  Let $k\subset K$ be a field extension. We say that elements $\alpha_1,\ldots,\alpha_n$ of $K$ are \emph{algebraically independent} over $k$, or that the set $\{\alpha_1,\ldots,\alpha_n\}\subset K$ is algebraically independent over $k$, if
  $$
   P(\alpha_1,\ldots,\alpha_n) \neq 0
   $$
 for every polynomial $P \in k[X_1,\ldots,X_n]\setminus\{0\}$.  When $n=1$, we rather say that $\alpha_1$ is \emph{transcendental} over $k$.
\end{defi}

Some of the most traditional choices of fields $k$ and $K$ are:
\begin{center}
  \begin{tabular}{| c | c | c |}
    \hline
    & $k$ & $K$\\
    \hline
    arithmetic case & $\QQ$ & $\CC$ \\
    \hline
    functional case & $\CC$ & $\CC(\!(t_1,\ldots,t_m)\!)$\\
                              \hline
\end{tabular}
\end{center}
These give origin to the two main branches of transcendence theory: arithmetic transcendence (or transcendental number theory), and functional transcendence. Albeit essentially distinct, these two branches are subtly, and often mysteriously, intertwined. 

\begin{obs}
  One could also replace the field of formal Laurent series $\CC(\!(t_1,\ldots,t_m)\!)$ above by the field of meromorphic functions on some open domain of $\CC^m$. In the one variable case, it is common to consider $\CC(t)$ as the base field $k$. Our framework includes this one, since $f_1,\ldots,f_n \in \CC(\!(t)\!)$ are algebraically independent over $\CC(t)$ if and only if $f_0=t,f_1,\ldots,f_n$ are algebraically independent over $\CC$.
\end{obs}

Closely related to the notion of algebraic independence is the quantitative notion of transcendence degree of a field extension.

\begin{defi}
Let $k\subset K$ be a field extension. A subset $S$ of $K$ is algebraically independent over $k$ if every finite subset of $S$ is algebraically independent over $k$. The \emph{transcendence degree} of $K$ over $k$, denoted by $\trdeg_kK$, is the maximal cardinality of a subset of $K$ algebraically independent over $k$.
\end{defi}

For instance, let us assume that $K=k(\alpha_1,\ldots,\alpha_n)$. Then $\trdeg_kK \le n$. To assert that $\trdeg_kK\le n-1$ is equivalent to assert that there exists a non-trivial algebraic relation, with coefficients in $k$, between $\alpha_1,\ldots,\alpha_n$. If $1\le r \le n$, to assert that $\trdeg_kK\ge r$ is equivalent to assert that some subset of $r$ elements of $\{\alpha_1,\ldots,\alpha_n\}$ is algebraically independent over $k$.

\begin{ex}[Exponential function]\label{ex:exponentialfunc}
  Consider $t$ as a formal variable and let
  $$
  e^t = \sum_{m=0}^{\infty}\frac{t^m}{m!} \in \CC(\!(t)\!)
  $$
  be the exponential power series. Then $t$ and $e^t$ are algebraically independent over $\CC$. In other words, $e^t$ is transcendental over $\CC(t)$. We prove this by contradiction. If $e^{t}$ were algebraic over $\CC(t)$, then there would exist a minimal integer $d\ge 1$ for which there are polynomials $P_0,\ldots,P_d\in \CC[X]$ satisfying
$$
\sum_{j=0}^d P_j(t)e^{jt} = 0\text{.}
$$
By differentiating with respect to $t$ and by subtracting the resulting equation from $d$ times the original equation, we obtain
$$
\sum_{j=0}^d (P'_{j}(t) +(j-d)P_j(t))e^{jt}= 0\text{,}
$$
so that the leading coefficient is now $P_d'(t)$. By induction, we would obtain polynomials $Q_0,\ldots,Q_{d-1}\in \CC[X]$ (not all zero; check!) such that $\sum_{j=0}^{d-1}Q_j(z)e^{jt}\equiv 0$, thereby contradicting the minimality of $d$.
\end{ex}

In general, the notion of algebraic independence admits the following scheme-theoretic interpretation. Let $k\subset K$ be a field extension, and consider a $K$-point $\alpha = (\alpha_1,\ldots,\alpha_n)\in \AA^n_k(K)$. Then, $\alpha_1,\ldots,\alpha_n$ are algebraically independent over $k$ if and only if the image of
$$
\alpha: \Spec K \to \AA^n_k
$$
is dense in $\AA^n_k$ for the Zariski topology --- that is, if $Y\subset \AA^n_k$ is a closed $k$-subvariety such that $\alpha \in Y(K)$, then $Y=\AA^n_k$.

This point of view allows us to give concrete geometric significance to functional transcendence. For instance, algebraic independence of one variable functions corresponds to Zariski-density of parameterised curves. To fix ideas, let $U$ be a neighbourhood of $0\in \CC$, and let $\varphi_1,\ldots,\varphi_n$ be holomorphic functions on $U$. We thus obtain a holomorphic curve
$$
\varphi=(\varphi_1,\ldots,\varphi_n): U  \to \CC^n = \AA_{\CC}^n(\CC)\text{.}
$$
To say that $\varphi_1,\ldots,\varphi_n$ (seen as formal power series) are algebraically independent over $\CC$ is equivalent to say that the image of the curve $\varphi$ is Zariski-dense in $\AA^n_{\CC}$.

\begin{ex}[Exponential function, revisited]\label{ex:exporev}
  It follows from the above discussion that the transcendence of $e^t$ over $\CC(t)$ is equivalent to the Zariski-density of the image of the holomorphic curve
  \begin{align*}
\varphi: \CC \to \CC^2\text{, }\qquad             z\mapsto (z,e^z)\text{.}
  \end{align*}
  Geometrically, this can be proved as follows. Assume that there exists an irreducible algebraic curve $C\subset \AA^2_{\CC} = \Spec \CC[X,Y]$ containing the image of $\varphi$. Since $e^{2\pi in} = 1$ for every $n \in \ZZ$, $C$ intersects the line $V(Y-1)$ at an \emph{infinite} number of points: $(2\pi i n ,1)\in \CC^2$ for $n\in \ZZ$. As $C$ and $V(Y-1)$ are both irreducible, this is only possible if $C=V(Y-1)$ (by a weaker form of Bézout's theorem). This would imply that $e^z\equiv 1$, which is clearly absurd. 
\end{ex}

\subsection{Arithmetic transcendence and Diophantine approximation}

It is widely acknowledged that functional transcendence is easier than arithmetic transcendence. Indeed, while results concerning functional transcendence date back to the founding fathers of calculus (see \cite{andre15} Footnote 1, p. 2), the first transcendence proof of explicitly defined numbers appears in Liouville's landmark paper \cite{liouville51}.

The fundamental insight of Liouville was to establish a link between arithmetic transcendence and Diophantine approximation.

\begin{theorem}[Liouville]
  If a real number $\alpha$ is algebraic of degree $d>1$ over $\QQ$, then there exists $c=c(\alpha) >0$ such that
  $$
\left| \alpha - \frac{p}{q} \right| > \frac{c}{q^d}\text{.}
$$
for every rational number of the form $p/q$, with $p,q\in \ZZ$ coprime and $q>0$.
\end{theorem}

\begin{proof}
  Let $P \in \ZZ[X]$ be an irreducible polynomial of degree $d$ such that $P(\alpha)=0$, and take $p/q$ as above with $|\alpha-p/q|<1$. By considering the Taylor expansion of $P$ at $\alpha$, we obtain
  $$
|P(p/q)| = | \sum_{i=1}^d\frac{P^{(i)}(\alpha)}{i!} (p/q - \alpha)^i| \le M|p/q-\alpha|
$$
where $M = \sum_{i=1}^n|P^{(i)}(\alpha)/i!|$. Since $P$ is irreducible, we have $P(p/q)\neq 0$; since it is of degree $d$ and has integral coefficients, we have $q^dP(p/q) \in \ZZ\setminus\{0\}$, so that $|P(p/q)|\ge 1/q^d$. We can thus take $c= \min \{1, (2M)^{-1}\}$.
\end{proof}

For instance, the above result gives the transcendence of $\alpha = \sum_{n=0}^{\infty}10^{-n!}$.

\begin{obs}[`The fundamental theorem of transcendence']\label{rmk:ftt}
The seemingly innocuous observation, used in the above proof, that the absolute value of a non-zero integer is at least 1 is at the heart of virtually \emph{every} proof in arithmetic transcendence (cf. \cite{masser16}).
\end{obs}

Liouville's theorem can be regarded as a general transcendence criterion. Some other general algebraic independence criteria in terms of Diophantine approximation exist. Nesterenko's proof of his theorem on values of Eisenstein series, to be discussed below, relies on the following particular case of Philippon's sophisticated criteria \cite{philippon86}.

For a polynomial $P \in \CC[X_1,\ldots,X_n]$, we denote by $\|P\|_{\infty}$ the maximum of the absolute values of its coefficients.

\begin{theorem}[Philippon]\label{thm:philippon}
  Let $n\ge 2$ be an integer and $\alpha_1,\ldots,\alpha_n$ be complex numbers. Suppose that there exists an integer $r\ge 2$ and real constants $a>b>0$ such that, for every sufficiently large positive integer $d$, there exists a polynomial $Q_d \in \ZZ[X_1,\ldots,X_n]\setminus\{0\}$ of degree $\deg Q_d = O(d \log d)$ satisfying
  $$
   \log \|Q_d\|_{\infty} =O( d(\log d)^2)
   $$
   and
   $$
  -ad^r \le \log |Q_d(\alpha_1,\ldots,\alpha_n)| \le -bd^r\text{.}
  $$
  Then
  $$
\trdeg_{\QQ} \QQ(\alpha_1,\ldots,\alpha_n)\ge r-1
  $$
\end{theorem}

Geometrically, we may interpret the hypotheses of the above result as an approximation condition of $\alpha$ by hypersurfaces in $\AA^n_{\QQ}$ in terms of their degree and their `arithmetic complexity'.

\subsection{Schneider-Lang and Siegel-Shidlovsky}

The connection between transcendence and Diophantine approximation suggests a closer inspection on values of analytic functions. Indeed, properties of such functions such as growth conditions, and differential or functional equations, can provide additional tools to the study of the approximation properties of their values.

Historically, this vague idea culminated in two precise and general theorems, those of Schneider-Lang and Siegel-Shidlovsky, which deal with entire or meromorphic functions over $\CC$ satisfying some algebraic differential equation, a growth condition, a functional transcendence statement, and some hypotheses of arithmetic nature.

We next state both theorems and derive some consequences, without saying anything about their proofs. The idea here is simply to help the reader to put Nesterenko's theorem in perspective (see Section \ref{sec:problems} below).

\subsubsection{Schneider-Lang}

Given a real number $\rho>0$, we say that the \emph{order} of an entire function $f$ on $\CC$ is $\le \rho$ if there exist real numbers $a,b>0$ such that
$$
|f(z)| \le ae^{b|z|^{\rho}}
$$
for every $z \in \CC$. A meromorphic function on $\CC$ is of order $\le \rho$ if it can be written as a quotient of two entire functions of order $\le \rho$.

\begin{theorem}[Schneider-Lang, cf. \cite{waldschmidt74} Thm. 3.3.1]
  Let $\rho_1,\rho_2 >0$ be real numbers, $K\subset \CC$ be a number field, $n\ge 2$ be an integer, and $f_1,\ldots,f_n$ be meromorphic functions on $\CC$ such that the ring $K[f_1,\ldots,f_n]$ is stable under the derivation $\frac{d}{dz}$. Let us further assume that:
  \begin{enumerate}
    \item $f_1$ and $f_2$ are algebraically independent over $K$;
    \item $f_i$ is of order $\le \rho_i$, for $i=1,2$. 
    \end{enumerate}
    Then, if $S$ denotes the set of $\alpha \in \CC$ such that, for every $1\le j \le n$, $\alpha$ is not a pole of $f_j$ and $f_j(\alpha)\in K$, we have:
    $$
    \text{card}(S) \le (\rho_1+\rho_2)[K:\QQ]\text{.}
    $$
  \end{theorem}

  In essence, this theorem simply asserts that $f_1,\ldots,f_n$ can only take too many simultaneous algebraic values if there is an algebraic relation between them.

\begin{obs}\label{rmk:genSL}
The Schneider-Lang criterion also admits geometric generalisations replacing differential equations by algebraic foliations. See \cite{herblot} and \cite{gasbarri10}.
\end{obs}
  
  As a first corollary, we can recover the following classical result which brings together the pioneering works of Hermite and Lindemann on the transcendence of $e$ and $\pi$.

\begin{coro}[Hermite-Lindemann]\label{coro-HL}
  For every $z \in \CC\setminus\{0\}$, we have
  $$
  \trdeg_{\QQ}\QQ(z,e^{z})\ge 1\text{.}
  $$
\end{coro}

\begin{proof}
Let $f_1(z)=z$ and $f_2(z)=e^z$. Clearly, both $f_1$ and $f_2$ are of order $\le 1$ and the ring $\QQ[f_1,f_2]$ is stable under $\frac{d}{dz}$. We have already seen in Example \ref{ex:exponentialfunc} that $f_1$ and $f_2$ are algebraically independent over $\CC$. By absurd, suppose that there exists $\alpha \in \CC\setminus\{0\}$ such that both $\alpha$ and $e^{\alpha}$ are algebraic, and set $K=\QQ(\alpha,e^{\alpha})$. Then $S$ would contain the infinite set $\{n\alpha \mid n \in \ZZ\}$, thereby contradicting Schneider-Lang's theorem. 
\end{proof}

Taking $z=2\pi i$, we obtain the transcendence of $\pi$; taking $z=1$, we obtain the transcendence of $e$.

In the same spirit, we can prove Schneider's theorem characterising `bi-algebraic' points for the $j$-invariant, seen as a holomorphic function on the Poincaré upper half-plane $\mathbf{H} = \{\tau \in \CC \mid \Im \tau >0\}$.

\begin{coro}[Schneider]\label{coro:schneider}
  If $\tau \in \mathbf{H}$ is not quadratic imaginary, then
  $$
\trdeg_{\QQ} \QQ(\tau,j(\tau))\ge 1\text{.}
  $$
\end{coro}

Note that, if $\tau \in \mathbf{H}$ is quadratic imaginary, i.e., $\QQ(\tau)$ is an imaginary algebraic extension of $\QQ$ of degree 2, then it follows from the classical theory of complex multiplication of elliptic curves that $j(\tau)$ is algebraic (see, for instance, \cite{silverman} Chapter II).

\begin{proof}
  Suppose that $\tau$ and $j(\tau)$ are both algebraic. Since $j(\tau)$ is algebraic, which means that the elliptic curve corresponding to $\tau$ can be defined over the field of algebraic numbers $\overline{\QQ}\subset \CC$, there exists a lattice $\Lambda = \ZZ \omega_1 + \ZZ \omega_2 \subset \CC$ such that $\tau= \frac{\omega_2}{\omega_1}$ and $g_2(\Lambda),g_3(\Lambda)\in \overline{\QQ}$. Here, $g_2(\Lambda)$ and $g_3(\Lambda)$ are the `invariants' of the Weierstrass elliptic function $\wp_{\Lambda}$, so that we have the differential equation
  $$
\wp_{\Lambda}'(z)^2 = 4 \wp_{\Lambda}(z)^3 - g_2(\Lambda) \wp_{\Lambda}(z) - g_3(\Lambda)\text{.}
  $$

  Set
  $$
(f_1(z),f_2(z),f_3(z),f_4(z))= (\wp_{\Lambda}(z),\wp_{\Lambda}(\tau z),\wp_{\Lambda}'(z), \wp'_{\Lambda}(\tau z))
$$
and let $K$ be a number field containing $\tau$, $j(\tau)$, and the field of definition of all $2$-torsion points on the elliptic curve $E$ over $\QQ(j(\tau))$ such that $E(\CC) = \CC/\Lambda$. Then $K[f_1,f_2,f_3,f_4]$ is stable under derivation and
$$
f_j(n\omega_1 + \frac{1}{2}\omega_1) \in K
$$
for every $j=1,\ldots,4$ and $n\in \ZZ$. Since each $f_j$ is of finite order (cf. \cite{MR14} Chapter 10), it follows from Schneider-Lang's theorem that $f_1$ and $f_2$ cannot be algebraically independent over $K$. This implies that there exists $m \in \ZZ$ such that $m\omega_2$ is a period of $f_2(z) = \wp_{\Lambda}(\tau z)$, so that there exists $a,b\in \ZZ$ satisfying
$$
m\omega_2 = a \frac{\omega_1}{\tau} + b \frac{\omega_2}{\tau}\text{.}
$$
Since $\tau= \omega_2/\omega_1$, we get $m\tau^2 - b\tau - a =0$.
\end{proof}

\begin{exo}
The Gelfond-Schneider theorem asserts that for algebraic numbers $\alpha \not\in \{0,1\}$ and $\beta\not\in \QQ$, $\alpha^{\beta}$ is transcendental. Derive this result from Schneider-Lang's theorem. Hint: consider the functions $e^z$ and $e^{\beta z}$.
\end{exo}

\subsubsection{Siegel-Shidlovsky}

We say that a power series
$$
f(z) = \sum_{n=0}^{\infty}\frac{a_n}{n!}z^n
$$
defines a \emph{Siegel $E$-function} if:
\begin{enumerate}
\item there exists a number field $K\subset \CC$ such that $a_n \in K$ for every $n\ge 0$;
\item for every $\epsilon>0$, we have $\max_{\sigma}|\sigma(a_n)|=O(n^{\epsilon n})$, where $\sigma$ runs through the set of all field embeddings of $K$ into $\CC$;
  \item for every $\epsilon >0$, there exists a sequence of strictly positive numbers $(q_n)_{n\ge 0}$ such that $q_n = O(n^{\epsilon n})$ and $q_na_j$ is an algebraic integer of $K$ for every $0\le j \le n$. 
\end{enumerate}

The second condition above implies that $f$ defines an entire function on $\CC$. The prototype of an $E$-function is the exponential function $e^z$, but other remarkable examples include some special cases of hypergeometric functions, such as Bessel's function
$$
J_0(z) = \sum_{n=0}^{\infty}\frac{(-1)^n}{n!^2}\left(\frac{z}{2}\right)^{2n}\text{.}
$$

\begin{theorem}[Siegel-Shidlovsky; cf. \cite{baker75} Ch. 11]
  Let $n\ge 1$ be an integer and $f_1,\ldots,f_n$ be entire functions on $\CC$ whose Taylor coefficients at the origin all lie in a same number field $K\subset \CC$ and for which there exist rational functions $g_{ij} \in K(z)$, $1\le i,j\le n$, such that
  $$
\frac{df_i}{dz} = \sum_{j=1}^ng_{ij}f_j
$$
for every $1\le i \le n$. If, moreover:
\begin{enumerate}
\item $f_1,\ldots,f_n$ are algebraically independent over $K(z)$, and
  \item each $f_i$ is a Siegel $E$-function,
  \end{enumerate}
  then, for every non-zero algebraic number $\alpha \in \CC$ which is not contained in the set of poles of $g_{ij}$, we have
  $$
\trdeg_{K}K(f_1(\alpha),\ldots,f_n(\alpha))=n\text{.}
  $$
\end{theorem}

Besides the arithmetic and the growth conditions, Siegel-Shidlovsky's and Schneider-Lang's theorems differ in an essential aspect: in Siegel-Shidlovsky the differential equation must be linear over $K(z)$. This more restrictive hypothesis yields, on the other hand, a stronger result of algebraic independence, while in Schneider-Lang we can only obtain transcendence. Observe however that the hypotheses of both results are structurally very similar.

\begin{ex}
  The Bessel function $J_0$ introduced above satisfies the linear equation
  $$
z^2 \frac{d^2J_0}{dz^2} + z\frac{d J_0}{dz} + z^2J_0 =0\text{.}
$$
Applying Siegel-Shidlovsky's theorem, we obtain that for every algebraic $\alpha \in \CC$, $J_0(\alpha)$ and $J_0'(\alpha)$ are algebraically independent over $\QQ$.
\end{ex}

\begin{obs}
As Schneider-Lang's theorem (see Remark \ref{rmk:genSL} above), there are also geometric generalisations of Siegel-Shidlovsky's theorem; see \cite{bertrand12} and \cite{gasbarri13}.
\end{obs}

\section{The theorem of Nesterenko}

Both in Schneider-Lang and in Siegel-Shidlovsky results, functions (holomorphic or meromorphic) are assumed to be defined on the whole complex plane $\CC$. This rules out any immediate application of these methods to functions defined on proper domains of $\CC$, such as modular functions.

Note that Schneider's theorem (Corollary \ref{coro:schneider} above) is indeed a theorem on the values of some modular function, but its proof relies fundamentally on elliptic functions. Schneider himself, in his famous memoir \cite{schneider57} (p. 138),  asks if it is possible to recover his result through a \emph{direct} study of the $j$-function.

Schneider's question remains unanswered, but we dispose nowadays of other transcendence results on modular functions with truly modular proofs. The following statement was conjectured by Mahler \cite{mahler69a} and proved by Barré-Sirieix, Diaz, Gramain and Philibert \cite{stephanois}.

\begin{theorem}
  For every $\tau \in \mathbf{H}$,
  \begin{align*}
    \trdeg_{\QQ}\QQ(e^{2\pi i \tau}, j(\tau))\ge 1\text{.}
  \end{align*}
\end{theorem}

This implies that $j(\tau)$ is transcendental whenever $e^{2\pi i \tau}$ is algebraic. Observe the appearance of the `modular parameter' $q= e^{2\pi i \tau}$ instead of $\tau$. Shortly after, Nesterenko generalised the above result in his famous theorem on values of Eisenstein series \cite{nesterenko96}.

\begin{theorem}
 For every $\tau \in \mathbf{H}$, we have
 $$
 \trdeg_{\QQ}\QQ(e^{2\pi i \tau},E_2(\tau),E_4(\tau),E_6(\tau))\ge 3\text{.}
 $$
\end{theorem}

Taking, for instance, $\tau=i$, we obtain the algebraic independence of $e^{\pi},\pi, \Gamma(1/4)$ over $\QQ$. We refer to Nesterenko's original paper \cite{nesterenko96} for more applications. In Section \ref{sec:ellper} below, we shall interpret Nesterenko's result in terms of periods of elliptic curves.

\subsection{Nesterenko's $D$-property and a zero lemma}

Let $X$ be a smooth affine variety over $\CC$ equipped with a vector field $v \in \Gamma(X,T_{X/\CC})$. We say that a closed subvariety $Y$ of $X$ is \emph{$v$-invariant} if $v$ restricts to a vector field on $Y$. In other words, if $v$ is seen as a derivation on the ring of regular functions $\mathcal{O}_X(X)$, and if $I_Y$ denotes the ideal of $\mathcal{O}_X(X)$ corresponding to $Y$, then $Y$ is $v$-invariant if $v(I_Y)\subset I_Y$.

\begin{ex}\label{ex:vfexp}
  Consider the vector field
$$
v = \frac{\partial}{\partial x} + y\frac{\partial}{\partial y} 
$$
on $\AA^2_{\CC}= \Spec \CC[x,y]$. It is clear that $V(y)$ is a $v$-invariant subvariety of $\AA^2_{\CC}$. Let us prove that this is the only one. For any $v$-invariant subvariety $Y$, we can write $Y=V(P)$ where
$$
P(x,y) = P_0(x) + P_1(x)y + \cdots + P_n(x)y^n\text{, }\ \ \ \ P_n(x)\neq 0
$$
is an irreducible polynomial dividing
$$
v(P)=\frac{\partial P}{\partial x} + y \frac{\partial P}{\partial y}
$$
Since $\deg P = \deg_x P_n + n =  \deg v(P)$, we must have $v(P)=\lambda P$ for some constant $\lambda \in \CC\setminus\{0\}$; that is
$$
P'_j(x) + j P_j(x) = \lambda P_j(x) \ \ \ \ \ \ ( j=0,\ldots,n)
$$
It is not hard to conclude from our hypotheses that $\lambda=1$, $P_1(x)=P_1(0)$ is constant $\neq 0$, and $P= P_1(0)y$. In other words, $Y=V(y)$.
\end{ex}

\begin{defi}
  Let $X$ be a smooth affine variety over $\CC$, $v \in \Gamma(X,T_{X/\CC})$ be a vector field on $X$, and
  $$
  \varphi: U \subset \CC \to X(\CC)
  $$
  be a non-constant holomorphic integral curve\footnote{By this, we simply mean that the derivative of $\varphi$ at every $z\in U$ is a multiple of $v_{\varphi(z)}$.} of $v$ defined on some connected open neighbourhood $U$ of $0\in \CC$. We say that $\varphi$ satisfies \emph{Nesterenko's $D$-property} if there exists a constant $c>0$ such that, for every $v$-invariant closed subvariety $Y$ of $X$, there exists a regular function $f$ on $X$ vanishing on $Y$ and satisfying
  $$
\ord_{0}(f\circ \varphi) \le c\text{.}
  $$
\end{defi}

\begin{lemma}\label{lemma:closureinvariant}
For $X$, $v$ and $\varphi$ as in the above definition, let $Y$ be the Zariski closure of $\varphi(U)\subset X(\CC)$ (i.e., $Y$ is the smallest subvariety of $X$ such that $\varphi(U)\subset Y(\CC)$). Then $Y$ is $v$-invariant.
\end{lemma}

\begin{proof}
  Let $f$ be a regular function on $X$ vanishing on $Y$; we must prove that $g\defeq v(f)$ also vanishes on $Y$. Let $\lambda$ be the holomorphic function on the open subset of $U$ where $v_{\varphi(z)}\neq 0$ such that $\varphi'(z)=\lambda(z) v_{\varphi(z)}$, so that
  $$
\lambda(z) (g\circ \varphi)(z) = (f\circ \varphi)'(z) =0\text{.}
$$
Since $U$ is connected and $\varphi$ is not constant, we obtain $g\circ \varphi \equiv 0$. Since the image of $\varphi$ is dense in $Y$, we conclude that $g$ vanishes on $Y$.
\end{proof}

The above lemma implies in particular that any $\varphi$ satisfying Nesterenko's $D$-property must have Zariski-dense image in $X$. 

Conversely, if the image of $\varphi$ is Zariski-dense in $X$ and if there exists a finite number of subvarieties $Y_1,\ldots,Y_m\subset X$ for which every $v$-invariant subvariety $Y$ of $X$ containing $\varphi(0)$ is contained in some $Y_i$, then $\varphi$ satisfies the $D$-property. This is how Nesterenko's $D$-property is verified in practice.

\begin{ex}
We have already seen in Example \ref{ex:exporev} that the image of the integral curve
\begin{align*}
  \varphi: \CC \to \AA^2_{\CC}(\CC)\text{, }\qquad             z \mapsto (z,e^z)\text{.}
\end{align*}
of the vector field $v$ of Example \ref{ex:vfexp} is Zariski-dense in $\AA^2_{\CC}$. Since $v$ admits at most a finite number of invariant subvarieties (actually, there is only one!), we conclude that $\varphi$ satisfies the $D$-property.
\end{ex}

\begin{obs}
  Alternatively, one could remark that the Zariski-density of the image of $z\mapsto (z,e^z)$ implies the Zariski-density of the image of
  $$
\varphi_a: z\mapsto (z+a,e^z)
$$
for any $a \in \CC$. This `stronger' statement immediately implies Nesterenko's $D$-property for $\varphi$ since \emph{any} leaf of the holomorphic foliation on $D(y) = \AA^2_{\CC}\setminus V(y)$ induced by $v$ can be parameterised by some $\varphi_a$ (cf. Lemma \ref{lemma:closureinvariant}).
\end{obs}

\begin{obs}
A famous theorem of Jouanolou (see, for instance, \cite{ghys00}) implies that any vector field $v$ on a smooth algebraic surface $X$ admitting a holomorphic integral curve with Zariski-dense image has at most a finite number of $v$-invariant subvarieties. Thus, when $\dim X=2$, Nesterenko's $D$-property is actually equivalent to the image of $\varphi$ being Zariski-dense in $X$.
\end{obs}

The \emph{raison d'être} of the $D$-property is that it gives a sufficient condition to an integral curve to satisfy certain `zero estimates' which are useful in Diophantine approximation. Here is a precise statement.

\begin{theorem}[Zero Lemma]\label{thm:zerolemma}
  Let $X$ be an open affine subscheme of $\AA^n_{\CC}$, $v \in \Gamma(X,T_{X/\CC})\setminus\{0\}$ be a vector field on $X$, $U\subset \CC$ be a neighbourhood of $0$, and $\varphi : U \to X(\CC)$ be a holomorphic map satisfying the differential equation
  $$
z \frac{d\varphi}{dz} = v \circ \varphi\text{.}
$$
If $\varphi$ satisfies the $D$-property, then there exists a constant $C>0$ such that, for every polynomial $P \in \CC[x_1,\ldots,x_n] \setminus\{0\}$, we have
$$
\ord_{0}(P\circ \varphi)\le C (\deg P)^n\text{.}
$$
\end{theorem}

The above result was proved in this geometric form by Binyamini \cite{binyamini14} and is based on Nesterenko's original result in \cite{nesterenko96}, Paragraph 5. It also admits more general versions (see \cite{fonseca19}, Appendix B).

Binyamini's approach is based on intersection theory of analytic cycles. We isolate the main technical details in the form of the following lemma.

\begin{lemma}[cf. \cite{binyamini14}]\label{lemma:binyamini}
  With notation as in Theorem \ref{thm:zerolemma}, there exists an additive function $\mult_{\varphi}$, the \emph{intersection multiplicity with $\varphi$ at $p=\varphi(0)$}, which takes an effective algebraic cycle $Z$ of $X$ and associates a natural number (or $+\infty$)
  $$
\mult_{\varphi}(Z) \in \NN \cup\{+\infty\}
$$
satisfying the following properties:
\begin{enumerate}
\item $\mult_{\varphi}(Z)$ only depends on the analytic germ of $Z$ at $p$;
\item If $Z=V(P)\cap X$, for some $P \in \CC[x_1,\ldots,x_n]\setminus\{0\}$, then $\mult_{\varphi}(Z) = \ord_0 (P \circ \varphi)$;
\item If $Z=p$, then $\mult_{\varphi}(Z)=1$;
\item For any closed subvariety $Y$ of $X$ for which $p\in Y$, and any polynomial $P \in \CC[x_1,\ldots,x_n]\setminus\{0\}$ vanishing identically on $Y$, we have $\mult_{\varphi}(Y) \le \ord_0 (P\circ \varphi)\cdot \mult_p(Y)$;\footnote{Here, $\mult_p(Y)$ denotes the Samuel multiplicity of the variety $Y$ at the closed point $p \in Y$. It is given by $S(T) = \frac{\mult_p(Y)}{d!}T^d + O(T^{d-1})$, where $d = \dim Y$, and $S \in \QQ[T]$ is the unique polynomial such that $S(n) = \text{length}(\mathcal{O}_{Y,p}/\mm_p^{n+1})$ for every $n \in \NN$.}
\item For any closed subvariety $Y$ of $X$, and any polynomial $P \in \CC[x_1,\ldots,x_n]\setminus\{0\}$  vanishing identically on $Y$ for which $v(P)$ does not vanish identically on $Y$, we have $\mult_{\varphi}(Y) \le \mult_{\varphi}(Y \cdot V(v(P)))$, where $Z_1\cdot Z_2$ denotes the intersection product of algebraic cycles;
  \item There is an integer $n_0\ge 0$ such that, for every closed subvariety $Y$ of $X$ not contained in a $v$-invariant subvariety of $X$, if $d\ge 1$ is the smallest  integer for which there exists $P \in \CC[x_1,\ldots,x_n]\setminus\{0\}$ of degree $d$ vanishing identically on $Y$, then $\min \{n  \mid v^n(P) = v(v(\cdots(v(P))\cdots)) \text{ does not vanish identically on }Y\}\le n_0$.
\end{enumerate}
\end{lemma}

Let us now sketch Binyamini's argument assuming the above lemma.

\begin{proof}[Proof of Theorem \ref{thm:zerolemma}]
  Let $P \in \CC[x_1,\ldots,x_n]\setminus\{0\}$ be a polynomial of degree $d \ge 1$. We want to show that $\ord_0(P\circ \varphi)\le C d^n$ for some constant $C>0$ not depending on $d$ or $P$.

  Set $Z^1 = V(P)\cap X$. The idea is to construct, by induction, cycles $Z^k$ of codimension $k$, for $2\le k \le n$, satisfying
  \begin{enumerate}[(i)]
  \item $\mult_{\varphi}(Z^k) \le \mult_{\varphi}(Z^{k+1}) + c\deg Z^k$, where $c$ is the constant of the $D$-property, and
    \item for every $1\le k \le n-1$, $\deg Z^{k+1}\le (d + c_0)\deg Z^k$, for some constants $c_0,c_1>0$ (not depending on $d$ or $P$).
    \end{enumerate}
    Once this is done, we have:
    \begin{align*}
      \ord_0(P \circ  \varphi) &= \mult_{\varphi}(Z^1)\\
                               &\le \mult_{\varphi}(Z^2) + c\deg Z^1\\
                               &\le \mult_{\varphi}(Z^3) + c(\deg Z^2 + \deg Z^1)\\
                               &\vdots\\
                               &\le \mult_{\varphi}(\underbrace{Z^n}_{0-cycle}) + c(\deg Z^{n-1} + \cdots + \deg Z^1)\\
                               &\le c(\deg Z^n + \deg Z^{n-1} + \cdots + \deg Z^1)\\
                               &\le c((d + c_0)^n + (d+c_0)^{n-1} + \cdots + (d+c_0))\\
      &\le cn(1 + c_0)^nd^n\text{,}
    \end{align*}
    so that we may take $C= cn(1+c_0)^n$.

    Let us now see how the sequence of cycles $Z^k$ is constructed, and where the $D$-property comes in. By induction, suppose that $Z^{k}$ has been constructed; let us construct $Z^{k+1}$. We write
    $$
Z^k = Z^k_n + Z^k_c
$$
where, by definition, the irreducible components of $Z^{k}_c$ are those of $Z^k$ which are contained in some $v$-invariant subvariety of $X$. Now, write $Z_n^k= \sum_i m_iY_i$, and, for every $i$, let $d_i\ge 1$ be the smallest integer for which there exists a polynomial $P_i \in \CC[x_1,\ldots,x_n]\setminus\{0\}$ of degree $d_i$ vanishing identically on $Y_i$ (note that $d_i\le d$). Let $n_i = \min \{n \mid v^n(P_i)|_{Y_i}\not\equiv 0\}$.  Since, by definition, $Y_i$ is not contained in a $v$-invariant subvariety, $n_i$ is finite. Then we define
$$
Z^{k+1} = \sum_i m_i Y_i \cdot V(v^{n_i}(P_i))\text{.}
$$
Finally, (i) follows from properties 4 (combined with the $D$-property) and 5, and (ii) follows from 6.
\end{proof}

\begin{exo}
Write down a complete proof for the Zero Lemma in dimension 2.
\end{exo}

\subsection{Mahler's theorem and the $D$-property for the Ramanujan equations}

We shall now study the foliation induced by the Ramanujan equations and prove the corresponding $D$-property.

We start with Mahler's theorem asserting that the $j$-invariant satisfies no algebraic differential equation of second order or lower (see \cite{mahler69}). More precisely, if we denote
$$
\theta = \frac{1}{2\pi i}\frac{d}{d\tau}\text{,}
$$
Mahler proved that the holomorphic functions on $\mathbf{H}$
$$
\tau,e^{2\pi i \tau}, j(\tau), \theta j(\tau), \theta^2j(\tau)
$$
are algebraically independent over $\CC$.

\begin{lemma}
  We have
  $$
\QQ(j, \theta j, \theta^2j) = \QQ(E_2,E_4,E_6)\text{.}
  $$
\end{lemma}

\begin{proof}
  Since $j \in \QQ(E_4,E_6)$, it follows immediately from Ramanujan's equations that $\QQ(j, \theta j, \theta^2j) \subset \QQ(E_2,E_4,E_6)$. Explicitly:
  $$
j = 1728\frac{E_4^3}{E_4^3-E_6^2}\text{, } \theta j = -1728\frac{E_4^2E_6}{E_4^3-E_6^2}\text{, } \theta^2j = 288 \frac{-E_2E_4^2E_6 + 4E_4E_6^2+3E_4^4}{E_4^3-E_6^2}\text{.}
$$
The above formulas can be inverted. Recall that $\Delta = \frac{1}{1728}(E_4^3-E_6^2)$ and that $\theta \log \Delta = E_2$ (this follows from Ramanujan's equations). Writing $j = E_4^3/\Delta$ and $j-1728 = E_6^2/\Delta$, and using the Ramanujan equations, we get
$$
\theta \log j = -\frac{E_6}{E_4}\text{, } \theta \log (j-1728) = - \frac{E_4^2}{E_6}
$$
so that
$$
E_4 = \theta \log j\cdot \theta \log (j-1728)\text{, }E_6 = - (\theta \log j)^2 \cdot \theta \log (j-1728) \in \QQ(j,\theta j)\text{.}
$$
Finally,
$$
E_2 = \theta \log \Delta = 3 \theta \log E_4 - \theta \log j \in \QQ(j,\theta j, \theta^2j)\text{.}
$$
\end{proof}

It follows from the above lemma that Mahler's theorem is equivalent to the following statement.

\begin{theorem}\label{thm:mahler}
  The holomorphic functions on $\mathbf{H}$
  $$
\tau, e^{2\pi i \tau}, E_2(\tau), E_4(\tau), E_6(\tau)
$$
are algebraically independent over $\CC$.
\end{theorem}

Our proof is different from Mahler's, but it is still fairly elementary. It relies on the following simple geometric considerations (see Remark \ref{rmk:motivationhre} below for the original motivation). Let $T$ be the open affine subscheme of $\AA^3_{\CC} = \Spec \CC[t_1,t_2,t_3]$ where $t_2^3-t_3^2 \neq 0$ and consider the surjective map
\begin{align*}
  \pi : T \to \AA^1_{\CC}\text{, }    \qquad    (t_1,t_2,t_3) \mapsto 1728 \frac{t_2^3}{t_2^3-t_3^2}\text{.}
\end{align*}
If $G$ denotes the subgroup scheme of $\SL_{2,\CC}$ of upper triangular matrices, so that
$$
G(\CC) = \left\{\left.\left(\begin{array}{cc}
                        x^{-1} & y \\
                        0 & x
                      \end{array}\right)\in {\SL}_2(\CC) \right| x \in \CC^{\times}, y \in \CC\right\}\text{,}
$$
then $G$ acts on $T$ by
$$
(t_1,t_2,t_3)\cdot \left(\begin{array}{cc} x^{-1} & y \\
                         0 & x\end{array}\right) = (-12xy + x^{2}t_1, x^{4}t_2, x^{6}t_3)\text{.}
$$
This action clearly preserves the fibres of $\pi$. In fact, $T$ is `almost' a $G$-torsor over $\AA^1_{\CC}$.

\begin{lemma}\label{lemma:finsurj}
  For every $z \in \CC$ and $t \in T(\CC)$ such that $\pi(t) = z$, the morphism
  \begin{align*}
    G \to \pi^{-1}(z)\text{, }\qquad    g \mapsto t\cdot g
  \end{align*}
  is finite and surjective.
\end{lemma}

\begin{proof}
Exercise.
\end{proof}

We are now ready for our proof.

\begin{proof}
 Set $X = \AA^2_{\CC}\times T$, and denote $p = \pi \circ \pr_2 : X \to \AA^1_{\CC}$. We must prove that the image of the holomorphic map
  \begin{align*}
    \varphi : \mathbf{H} \to X(\CC)\text{, } \qquad   \tau \mapsto (\tau, e^{2\pi i \tau}, E_2(\tau), E_4(\tau), E_6(\tau))
  \end{align*}
  is Zariski-dense in $X$. Note that $\varphi$ is well defined since Ramanujan's $\Delta$ function
  $$
\Delta(\tau) = \frac{E_4(\tau)^3- E_6(\tau)^2}{1728}
$$
never vanishes on $\mathbf{H}$; moreover, $(p\circ \varphi)(\tau) = j(\tau)$ for every $\tau \in \mathbf{H}$.

Since $p$ and $j$ are surjective, it suffices to prove that the image of $\varphi$ is Zariski-dense in every fibre of $p$. It follows from Lemma \ref{lemma:finsurj} that, for every $\tau \in \mathbf{H}$, the map
\begin{align*}
  f_{\tau}: \AA^2_{\CC}\times G \to p^{-1}(j(\tau))\text{, }  \qquad  (a,b,g)\mapsto (a,b, (E_2(\tau),E_4(\tau),E_6(\tau))\cdot g)
\end{align*}
is finite and surjective, so that $\varphi(\mathbf{H})\cap p^{-1}(j(\tau))$ is Zariski-dense in $p^{-1}(j(\tau))$ if and only if $f_{\tau}^{-1}(\varphi(\mathbf{H})\cap p^{-1}(j(\tau)))$ is Zariski-dense in $\AA^2_{\CC}\times G$.

Using the (quasi)modularity of $E_2$, $E_4$, and $E_6$, one easily verifies that $f_{\tau}^{-1}(\varphi(\mathbf{H})\cap p^{-1}(j(\tau)))$ contains the set
$$
S_{\tau} = \left\{(\gamma \cdot \tau, e^{2\pi i \gamma \cdot \tau}, g_{\gamma,\tau}) \in \CC^2\times G(\CC) \left| \gamma = \left(\begin{array}{cc} a & b \\ c & d \end{array} \right) \in {\SL}_2(\ZZ)\text{, } c\tau+ d \neq 0 \right\}\right.
$$
where
$$
g_{\gamma,\tau} = \left(\begin{array}{cc}
                          (c\tau + d)^{-1} & - c/2\pi i \\
                          0 & c\tau+d
                        \end{array}\right) \in G(\CC)\text{,}
                      $$
                      so that it suffices to prove that $S_{\tau}$ is Zariski-dense in $\AA^2_{\CC}\times G$ for any $\tau \in \mathbf{H}$.

As a first reduction, observe that it suffices to prove that the set
$$
\{(e^{2\pi i \gamma \cdot \tau}, g_{\gamma,\tau}) \in \CC \times G(\CC) \mid \gamma \in {\SL}_{2}(\ZZ)\text{, }c\tau + d \neq 0\}                     
$$
is Zariski-dense in $\AA^1_{\CC}\times G$. Indeed, if we denote ${}_n\gamma = \left(\begin{array}{cc}1 & n \\ 0 & 1 \end{array} \right)\cdot \gamma$ for every $\gamma \in \SL_{2}(\ZZ)$ and $n\in \ZZ$, then
$$
(\gamma_n\cdot \tau,e^{2\pi i {}_n\gamma \cdot \tau}, g_{{}_n\gamma,\tau}) = (\gamma\cdot \tau + n, e^{2\pi i \gamma \cdot \tau}, g_{\gamma,\tau})
$$
and our claim follows from the Zariski-density of $\ZZ\subset \CC$ in $\AA^1_{\CC}$.

We now perform a second reduction: it suffices to prove that the set
$$
\{(e^{2\pi i \frac{a}{c}},c) \in \CC^2 \mid (a,c)\in \ZZ^2\text{, }\text{gcd}(a,c)=1\}
$$
is Zariski-dense in $\AA_{\CC}^2$. Indeed, let $P \in \CC[t,x,y]\setminus\{0\}$ be such that
$$
P(e^{2\pi i \gamma \tau}, g_{\gamma,\tau}) = P(e^{2\pi i \gamma \tau}, c\tau + d, -c/2\pi i) = 0
$$
for every $\gamma \in \SL_{2}(\ZZ)$. Writing $P = \sum_{j=0}^NP_j(t,y)x^j$, with $P_{N} \neq 0$, we obtain that
$$
\sum_{i=0}^NP_j(e^{2\pi i \gamma_n\cdot \tau}, - c/2\pi i)(c\tau + d + cn)^j =0
$$
for every $\gamma \in \SL_{2}(\ZZ)$ and $n \in \ZZ$, where $\gamma_n = \gamma\cdot  \left(\begin{array}{cc}1 & n \\ 0 & 1 \end{array} \right)$. We can assume $c\neq 0$. Multiplying the above equation by $(c\tau+d + cn)^{-N}$ and letting $n \to +\infty$, we obtain
$$
P_N(e^{2\pi i \frac{a}{c}}, -c/2\pi i) =0
$$
for every $\gamma \in \SL_{2}(\ZZ)$, and our claim follows.

Finally, suppose that there exists a polynomial $P(x,y) = \sum_{j=0}^NP_j(y)x^j \in \CC[x,y]$, such that $P(e^{2\pi i \frac{a}{c}},c)=0$ for every $(a,c)\in \ZZ^2$ with $\text{gcd}(a,c)=1$. Taking, for instance, $c$ to be a prime  $p\ge N+1$, we see that $P(x,p)$ is a polynomial of degree $N$ having at least $N+1$ roots: $1,e^{2\pi i \frac{1}{p}},\ldots,e^{2\pi i\frac{p-1}{p}}$, so that $P(x,p)=0$. As there are infinitely many prime numbers greater than $N+1$, we conclude that $P=0$.   
\end{proof}

\begin{obs}\label{rmk:motivationhre}
  The affine space $T$ can be identified to the moduli space of isomorphism classes of complex elliptic curves $E$ endowed with a basis $b=(\omega,\eta)$ of $H^1_{\dR}(E)$ such that $\omega \in H^0(E,\Omega^1_{X/\CC})$ and $\langle \omega, \eta \rangle=1$ (where $\langle \ , \ \rangle $ denotes the cup product on algebraic de Rham cohomology) in a way that the holomorphic map $\tau\mapsto (E_2(\tau),E_4(\tau),E_6(\tau))$ corresponds to
\begin{align*}
   \varphi: \mathbf{H} \to T(\CC)\text{, }\qquad \tau \mapsto [(\CC/\ZZ + \tau \ZZ,(\omega_{\tau},\eta_{\tau})]
\end{align*}
where $\omega_{\tau}=2\pi i dz$ and $\eta_{\tau}= \frac{1}{2\pi i}\wp_{\tau}(z)dz - \frac{E_{2}(\tau)}{12}2\pi i dz$ (cf. \cite{fonseca18} Section 8). Under such moduli-theoretic interpretation, the group scheme $G$ acts by right multiplication on the basis $b$ seen as a row vector. Note that $T$ admits a natural map to the moduli stack of complex elliptic curves $\mathcal{M}_{1,1}$ and that $T$ is a \emph{bona fide} $G$-torsor over $\mathcal{M}_{1,1}$. In the above proof, we replaced $\mathcal{M}_{1,1}$ by its coarse moduli scheme $\AA^1_{\CC}$, at the cost of weakening this `torsor property' (cf. Lemma \ref{lemma:finsurj}). 
\end{obs}

\begin{coro}\label{coro:every-leaf}
Every leaf of the holomorphic foliation on $T(\CC)$ induced by the vector field
$$
v = \frac{(x_1^2-x_2)}{12}\frac{\partial}{\partial x_1} + \frac{(x_1x_2-x_3)}{3}\frac{\partial}{\partial x_2} + \frac{x_1x_3-x_2^2}{2}\frac{\partial}{\partial x_3}
$$
is Zariski-dense in $T$.
\end{coro}

\begin{proof}
Let $(c,d) \in \CC^2\setminus\{0\}$ and define
\begin{align*}
\varphi_{c,d}(\tau) = \left((c\tau+d)^2E_2(\tau)+ \frac{12c}{2\pi i}(c\tau +d), (c\tau+d)^4E_4(\tau),(c\tau+d)^6E_6(\tau)\right)\text{.}
\end{align*}
One may easily check that $\varphi_{c,d}$ satisfies the differential equation
$$
\theta\varphi_{c,d} = (c\tau+d)^{-2}v\circ \varphi_{c,d}
$$
so that its image is a leaf of the foliation defined by $v$.

By Theorem \ref{thm:mahler}, the image of each $\varphi_{c,d}$ is Zariski-dense in $T$. Thus, to finish our proof it suffices to show that any point of $T$ lies in the image of $\varphi_{c,d}$ for some $(c,d) \in \CC^2\setminus\{0\}$. Let $t \in T(\CC)$ and choose $\tau \in \mathbf{H}$ such that $\pi(t) = j(\tau)$, so that $t$ and $(E_2(\tau),E_4(\tau),E_6(\tau))$ lie in the same $\pi$-fibre. By Lemma \ref{lemma:finsurj}, there exists $g\in G(\CC)$ such that
$$
t = (E_2(\tau),E_4(\tau),E_6(\tau))\cdot g\text{.}
$$
To conclude, we simply remark that any element of $G$ is of the form
$$
g = \left(\begin{array}{cc}
            (c\tau+d)^{-1} & -c/2\pi i\\
            0 & c\tau + d
            \end{array}\right) \in G(\CC)
          $$
          for some $(c,d) \in \CC^2\setminus\{0\}$, so that
          $$
t = (E_2(\tau),E_4(\tau),E_6(\tau))\cdot g = \varphi_{c,d}(\tau)\text{.}
          $$
\end{proof}

Finally, for the next theorem, we consider $E_2$, $E_4$, $E_6$ as functions of $q \in D = \{z \in \CC \mid |z|<1\}$ under the change of variables $q=e^{2\pi i \tau}$. Note that $\theta = q\frac{d}{d q}$.

\begin{theorem}
  Consider the vector field $w$ on $\AA^4_{\CC}$ given by
  $$
w = x_0 \frac{\partial}{\partial x_0} + \frac{(x_1^2-x_2)}{12}\frac{\partial}{\partial x_1} + \frac{(x_1x_2-x_3)}{3}\frac{\partial}{\partial x_2} + \frac{x_1x_3-x_2^2}{2}\frac{\partial}{\partial x_3}
$$
and consider the holomorphic curve
\begin{align*}
  \varphi: D \to \AA^4(\CC)\text{, } \qquad    q \mapsto (q,E_2(q),E_4(q),E_6(q))
\end{align*}
satisfying the differential equation
$$
\theta \varphi = w \circ \varphi\text{.}
$$
Then $\varphi$ satisfies Nesterenko's $D$-property.
\end{theorem}

\begin{proof}
  Since we already know from Mahler's theorem that the image of $\varphi$ is Zariski-dense in $\AA^4_{\CC}$, it is sufficient to prove that there's only a finite number of maximal $w$-invariant subvarieties containing $\varphi(0)=(0,1,1,1)$. Actually, we shall prove that every $w$-invariant subvariety containing $p$ is either $V(x_0)$ or it is contained in $V(x_2^3-x_3^2)$.

  Consider the projection
  $$
  \pi: \AA^4_{\CC} \to \AA^3_{\CC}\text{, }\qquad (x_0,x_1,x_2,x_3)\mapsto (x_1,x_2,x_3)\text{.}
  $$
  Let $Y\subset \AA^4_{\CC}$ be a $w$-invariant subvariety containing $p$. Then $\pi(Y)\neq \emptyset$ is $v$-invariant, where $v$ is the `Ramanujan vector field' of Corollary \ref{coro:every-leaf}. It follows from this corollary that either $\pi(Y) \subset V(x_2^3-x_3^2)$ in $\AA^3_{\CC}$, or $\pi(Y)=\AA^3_{\CC}$. In the first case, we have that $Y\subset V(x_2^3-x_3^2)$ in $\AA^4_{\CC}$. To conclude, we only need to prove that $\pi(Y) = \AA^3_{\CC}$ implies that $Y=V(x_0)$.

  Now, if $\pi(Y)= \AA^3_{\CC}$, then $Y$ has dimension at least 3, so that $Y=V(P)$ for a unique irreducible polynomial
  $$
P = f_nx_0^n + \cdots + f_1x_0 + f_0
$$
with $f_i \in \CC[x_1,x_2,x_3]$, and $f_n$ monic. Since $Y$ is $w$-invariant, there exists $Q \in \CC[x_0,x_1,x_2,x_3]$ such that $w(P)=QP$. By considering the degree in $x_0$, we see that $Q=g \in \CC[x_0,x_1,x_3]$, and we get the equations
$$
v(f_j) = (g-j)f_j
$$
for every $j=0,\ldots,n$. Again using Corollary \ref{coro:every-leaf}, we conclude (exercise) that $f_1=1$ and $f_j=0$ for every $j\neq 1$, i.e., $P=x_0$.
\end{proof}

By combining the above theorem with the Zero Lemma (Theorem \ref{thm:zerolemma}), we obtain the following corollary.

\begin{coro}\label{coro:zlramanujan}
  There exists a constant $C>0$ such that, for every polynomial $P \in \CC[x_0,x_1,x_2,x_3]\setminus\{0\}$, we have
  $$
\ord_{q=0}P(q,E_2(q),E_4(q),E_6(q))\le C (\deg P)^4\text{.}
  $$
\end{coro}

\subsection{Sketch of Nesterenko's proof}\label{subsec:sketch}

Nesterenko's proof is rather long and intricate, but its general structure is not difficult to understand. In what follows we outline the main steps of Nesterenko's approach.

We see $E_2$, $E_4$, $E_6$ as holomorphic functions on the unit disk through the change of variables $q = e^{2\pi i \tau}$. To avoid any confusion, we adopt the following convention: $q$ will denote a generic variable in $D$, and $z \in D$ a point, so that $E_{2k}(q)$ is a function and $E_{2k}(z)$ is a complex number.

Let $z \in D\setminus\{0\}$. We want to prove that
\begin{align*}
\trdeg_{\QQ} \QQ(z,E_2(z),E_4(z),E_6(z)) \ge 3
\end{align*}
and the main idea is to apply Philippon's algebraic independence criterium as stated in Theorem \ref{thm:philippon} above. This means that we have to construct a sequence of polynomials $Q_d \in \ZZ[x_0,\ldots,x_3]$, $d\gg 0$, satisfying
\begin{enumerate}
\item $\deg Q_d =O( d \log d)$,
\item $\log \|Q_d\|_{\infty} =O( d (\log d)^2)$, and 
  \item $-ad^4 \le \log |Q_d(z,E_2(z),E_4(z),E_6(z)| \le -bd^4$
  \end{enumerate}
for some constants $a,b>0$ (not depending on $d$). The first step in obtaining the sequence $(Q_d)_{d\gg 0}$ is the construction of the so-called \emph{auxiliary polynomials}.

\begin{lemma}\label{lemma:lemma1}
  There is a constant $c >0$ such that, for every integer $d\gg 0$ there exists a polynomial $P_d \in \ZZ[x_0,\ldots,x_3]$ satisfying
  \begin{enumerate}
  \item $\deg P_d = d$,
  \item $\log \|P_d\|_{\infty} = O(d\log d)$, and
    \item $\ord_{q=0}P_d(q,E_2(q),E_4(q),E_6(q))\ge  c d^4$. 
  \end{enumerate}
\end{lemma}

\begin{proof}[Sketch of the proof]
  For every multi-index $J=(j_0,\ldots,j_3)$ with $|J| = j_0+\cdots + j_3\le d$, write
  $$
q^{j_0}E_2(q)^{j_1}E_4(q)^{j_2}E_6(q)^{j_3} =  \sum_{i\ge 0}t_{i,J}q^i \in \ZZ[\![q]\!]\text{.}
$$
Note that $t_{i,J}$ are rational integers because the Taylor coefficients of $E_{2k}$ are. Let
  $$
  P = \sum_{|J|\le d} v_J x^J  \in \ZZ[x_0,\ldots,x_3]
  $$
  be a polynomial with `unknown coefficients' $v_J$, so that
  $$
P(q,E_2(q),E_4(q),E_6(q)) = \sum_{i\ge 0}\left(\sum_{|J|\le d}t_{i,J}v_J \right)q^i \in \ZZ[\![q]\!]\text{.}
$$
Let $r= \lfloor\frac{1}{4!}d^4\rfloor$. To find $P$ such that $\ord_{q=0}P(q,E_2(q),E_4(q),E_6(q))\ge r$ is equivalent to solving the following $r$ linear equations over $\ZZ$ with $s = \binom{d+4}{4}$ variables:
$$
\sum_{|J|\le d}t_{i,J}v_J = 0 \text{, } \ \ \ \ i=0,\ldots,r-1\text{.}
$$
To get a bound on the size of a solution $(v_J)$, we apply the ubiquitous Siegel's lemma.

\begin{lemma}[Siegel; cf. \cite{waldschmidt74} Lemme 1.3.1 or \cite{MR14} Lemma 6.1]
Let $s>r>0$ be integers and $T \in M_{r\times s}(\ZZ)$ be such that $\|T\|_{\infty}\le b$. Then, there exists $v \in \ZZ^s\setminus\{0\}$ with $\|v\| \le 2(2sb)^{r/(s-r)}$ such that $Tv=0$.
\end{lemma}

To finish, we just need to prove that $\max_{i\le r-1,|J|\le d}|a_{i,J}| = O(d^d)$. This follows from the fact that the Taylor coefficients of $E_{2k}$, which are given up to a constant by the arithmetical function $\sigma_{2k-1}(m) = \sum_{d \mid m}m^{2k -1}$ have polynomial growth; namely, we have the trivial bound $\sigma_{2k-1}(m) \le m^{2k}$.
\end{proof}

\begin{exo}
Complete the above proof.
\end{exo}

For the next step, let us denote $f_d(q) = P_d(q,E_2(q),E_4(q),E_6(q))$ and $m_d = \ord_{q=0}f_d$.

\begin{lemma}\label{lemma:lemma2}
  There exist $\alpha>\beta>0$ and, for $d\gg 0$, a sequence $k_d =O(d \log m_d)$ satisfying
  $$
-\alpha m_d \le \log \left|f_d^{(k_d)}(z) \right| \le - \beta m_d\text{.}
  $$
\end{lemma}

This is the most technical, and most analytical, part of the proof. The main point in obtaining the bound $-\alpha m_d \le \log \left|f_d^{(k_d)}(z) \right|$ can explained through the following intuitive argument. If all the Taylor coefficients of $f_d$ at $q=z$ up to a sufficiently large order are too small, then its first non-zero Taylor coefficient at $q=0$ will have absolute value $<1$, thereby \emph{contradicting its integrality} (cf. Remark \ref{rmk:ftt}). Of course, the difficulty here consists in precisely quantifying `too small' and `sufficiently large order'. Once this is established, the other bound $\log \left|f_d^{(k_d)}(z) \right| \le - \beta m_d$ is a mere consequence of the Cauchy inequalities. We refer to \cite{NP01} Chapter 3, Lemma 3.3 for details.

Finally, we use the existence of the Ramanujan equations. Consider the derivation
$$
w=x_0 \frac{\partial}{\partial x_0} + \frac{(x_1^2-x_2)}{12}\frac{\partial}{\partial x_1} + \frac{(x_1x_2-x_3)}{3}\frac{\partial}{\partial x_2} + \frac{x_1x_3-x_2^2}{2}\frac{\partial}{\partial x_3}
$$
and define, for every $k\ge 1$,
$$
w^{[k]} =12^{k} w \circ(v-1) \circ \cdots \circ (w-(k-1))\text{.} 
$$

We set
$$
Q_d = w^{[k_d]}P_d \in \ZZ[x_0,\ldots,x_3]\text{.}
$$
Since $k_d =O( d \log m_d)$, it is clear that $\deg Q_d = O(d\log m_d)$, and that $\log \|Q_d\|_{\infty} = O(d(\log d)(\log m_d))$. Moreover, the identity of derivations
$$
q^k\frac{d^k}{dq^k} = \theta \circ (\theta -1 ) \circ \cdots \circ (\theta - (k-1))
$$
where $\theta = q\frac{d}{dq}$, implies that
$$
(12q)^{k_d}f_d^{(k_d)}(q) = Q_d(q,E_2(q),E_4(q),E_6(q))\text{.}
$$
Using Lemma \ref{lemma:lemma2}, we immediately deduce the bound
$$
-\alpha m_d - \gamma k_d \le \log |Q_d(z,E_2(z),E_4(z),E_6(z))|\le -\beta m_d - \gamma k_d
$$
where $\gamma = -\log |12z| >0$.

To conclude, we observe that $k_d = O(d\log m_d)$, that $d^4 = O(m_d)$ by Lemma \ref{lemma:lemma1}, and that $m_d = O(d^4)$ by Corollary \ref{coro:zlramanujan}.

\begin{obs}
The use of `auxiliary polynomials' is an essential step in Nesterenko's proof. Similar ideas occur in most results in Diophantine approximation or transcendental number theory. The reader may consult \cite{masser16} for a thorough exposition of the role of auxiliary polynomials in number theory. 
\end{obs}

\section{Periods}

Roughly speaking, a period (or an arithmetic period) is an integral that appears in algebraic geometry over $\QQ$.

The set of all periods forms a countable subset of $\CC$ -- actually, a subring -- containing all the algebraic numbers, but `most' of the periods are transcendental numbers. It is the subtle connection between these numbers and algebraic geometry that allows us to make sense of their structure and their symmetries.

Most of the general theory is still largely conjectural, but the omnipresence of periods in number theory (regulators, special values of $L$-functions, etc.) and high energy physics (Feynman amplitudes) makes the study of periods one of the most attractive subjects in current mathematics. 

\subsection{Definition} 

There are several equivalent definitions; we start with the more elementary ones given in \cite{KZ01}.

\begin{defi}
 A real number $\varpi$ is a \emph{period} if it can be written as an absolutely convergent integral
 $$
 \varpi = \int_{\sigma} \frac{P(t_1,\ldots,t_n)}{Q(t_1,\ldots,t_n)}dt_1\cdots dt_n
 $$
 where $P,Q \in \QQ[t_1,\ldots,t_n]$, $Q\neq 0$, and $\sigma \subset \RR^n$ is a domain given by polynomial inequalities with rational coefficients. A complex number is a period if its real and imaginary parts are periods.
\end{defi}

For instance,
$$
\pi = \int_{t_1^2+t_2^2 \le 1}dt_1 dt_2\text{ and } \log 2 = \int_{1\le t \le 2} \frac{dt}{t}
$$
are periods. A less trivial example is
$$
\zeta(3) = \sum_{n=1}^{\infty}\frac{1}{n^3} = \int_{0<t_1<t_2<t_3<1}\frac{dt_1dt_2dt_3}{(1-t_1)t_2t_3}\text{.}
$$

\begin{exo}
 The algebraic number $\sqrt{2}$ is a period, since it can be written as
 $$
 \sqrt{2} = \int_{t^2\le 2}\frac{dt}{2}
 $$
 Similarly, show that \emph{any algebraic number is a period}.
\end{exo}

One can show that we can replace rational numbers by algebraic numbers, and rational functions by algebraic functions (with algebraic coefficients) in the definition of periods.

More generally, let $\overline{\QQ}\subset \CC$ be the algebraic closure of $\QQ$ in $\CC$, and consider a tuple $(X,D,\omega,\sigma)$, where $X$ is a quasi-projective variety over $\overline{\QQ}$, $Y$ is a closed subvariety of $X$, $\omega \in \Gamma(X,\Omega^n_{X/\overline{\QQ}})$ is a closed algebraic $n$-form on $X$ vanishing on $Y$, and $\sigma \subset X(\CC)$ is a singular (topological) $n$-chain with boundary $\partial \sigma \subset Y(\CC)$.

\begin{prop}
 For every tuple $(X,Y,\omega,\sigma)$ as above, the number
$$
\int_{\sigma}\omega
$$
is a period. Conversely, every period is of this form.
\end{prop}

\begin{proof}
See \cite{HM-S17} 12.2.
\end{proof}

Yet another way of defining periods, which makes explicit their `motivic' nature, is as follows. For a pair $(X,Y)$ as above, we can consider its \emph{relative algebraic de Rham cohomology} (see \cite{HM-S17} Chapter 3)
$$
H^n_{\dR}(X,Y)
$$
which is a finite dimensional $\overline{\QQ}$-vector space, and the singular cohomology of $X(\CC)$ relative to the closed subspace $Y(\CC)\subset X(\CC)$ with $\QQ$-coefficients, also called \emph{relative Betti cohomology},
$$
H_B^n(X,Y)
$$
which is a finite dimensional $\QQ$-vector space. Then, a theorem of Grothendieck says that, after base change to $\CC$, there is a canonical comparison isomorphism between these two cohomology groups:
$$
\comp : H^n_{\dR}(X,Y)\tensor_{\overline{\QQ}}\CC \stackrel{\sim}{\to} H_B^n(X,Y)\tensor_{\QQ}\CC
$$
Now, a period is simply a number that appears as an entry of the matrix of $\comp$ with respect to some $\overline{\QQ}$-basis of $H^n_{\dR}(X,Y)$ and some $\QQ$-basis of $H_B^n(X,Y)$.

Note that $Y$ can be empty, in which case we simply denote $H^n_{\dR}(X)$ and $H^n_B(X)$. For the proof of the equivalences between all of these different definitions of periods, we refer to \cite{HM-S17} Section 12.2.

\subsection{Elliptic periods and values of quasimodular forms}\label{sec:ellper}

Let $E$ be an elliptic curve over $\overline{\QQ}$ given by a Weierstrass equation
$$
E: y^2 = 4x^3 - ux - v \qquad (u,v \in \overline{\QQ}\text{, } u^3 - 27v^2 \neq 0)
$$
and consider the following algebraic differential forms on $E$:
$$
\omega = \frac{dx}{y}\text{,} \ \ \ \eta = x\frac{dx}{y}\text{.}
$$
Note that $\omega$ is regular on $E$ (form of the first kind), while $\eta$ has a pole at infinity (of order 2) with vanishing residue (form of the second kind). We can see $(\omega,\eta)$ as a basis of the algebraic de Rham cohomology $H^1_{\dR}(E)$ under its classical identification with the space of forms of second kind modulo exact forms (see \cite{HM-S17} Chapter 14).

Let $(\gamma_1,\gamma_2)$ be any oriented $\ZZ$-basis of the first homology group $H_1(E(\CC);\ZZ)$. Here, `oriented' means that the intersection product $\gamma_1\cap \gamma_2 =1$. We can then consider the four complex numbers
$$
\omega_1 = \int_{\gamma_1}\omega \text{, }\ \ \  \omega_2 = \int_{\gamma_2}\omega\text{, }\ \ \ \eta_1 = \int_{\gamma_1}\omega\text{, }\ \ \  \eta_2 = \int_{\gamma_2}\eta\text{.}
$$
Classically, $\omega_1$ and $\omega_2$ are known as `periods' of $E$, while $\eta_1$ and $\eta_2$ are called `quasi-periods'. Under the above modern definition, they are all periods.

\begin{ex}
  Consider the elliptic curve $E: y^2 = 4x^3 - 4x$ and let $\gamma_1 \in H_1(E(\CC);\ZZ)$ be the class of the connected component of $E(\RR)$ containing the 2-torsion point $(1,0)$.
  
  \begin{center}
   \includegraphics[scale=0.4]{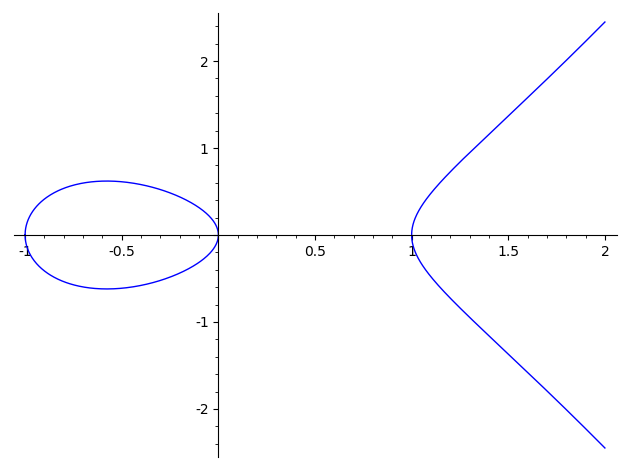}
  \end{center}
  
  Then
  $$
  \omega_1 = \int_{\gamma_1}\omega = 2\int_{1}^{\infty}\frac{dx}{\sqrt{4x^3-4x}} \stackrel{x=\frac{1}{t^2}}{=}2 \int_0^1\frac{dt}{\sqrt{1-t^4}}= \frac{\Gamma(1/4)^2}{2\sqrt{2\pi}}\text{,} 
  $$
  where to compute the last integral we used Euler's formula for the Beta function
  \begin{equation}\label{eq:betafunc}
B(a,b) \defeq \int_0^1t^{a-1}(1-t)^{b-1}dt = \frac{\Gamma(a)\Gamma(b)}{\Gamma(a+b)}\qquad (\Re(a),\Re(b)>0)\text{,}
  \end{equation}
the fact that $\Gamma(1/2) = \sqrt{\pi}$, and Euler's reflection formula:
$$
\Gamma(1/4)\Gamma(3/4) = \frac{\pi}{\sin(\pi/4)}\text{.}
$$
Let us remark $\omega_1$ computed above is a very classical object in the theory of elliptic integrals: it is half of the length of Bernoulli's lemniscate $(x^2+y^2)^2=x^2-y^2$.
\end{ex}

One of the fundamental problems in the theory of periods is to understand all the algebraic relations such numbers can satisfy. For instance, in the case of elliptic curves, we have Legendre's relation:
$$
\omega_1\eta_2-\omega_2\eta_1 = 2\pi i\text{,}
$$
which holds for every $E$. Are there any other relations? In particular, are these numbers transcendental? Are some of them algebraically independent?

\begin{conj}[Grothendieck's period conjecture for elliptic curves]
  With the above notation,
  \begin{align*}
  \trdeg_{\QQ}\QQ(\omega_1,\omega_2,\eta_1,\eta_2) =
  \begin{cases}
    2 & \text{if }E\text{ has complex multiplication}\\
    4 & \text{otherwise}
  \end{cases}
  \end{align*}
\end{conj}

Recall that $E$ has \emph{complex multiplication} if its endomorphism algebra $\End(E)$  strictly contains $\ZZ$. The geometric idea is that such extra endomorphisms of $E$ correspond to algebraic cycles (correspondences) on $E\times E$ which force algebraic relations between periods.

\begin{exo}\label{exo:cm}
  An endomorphism $\varphi:E \to E$ defined over $\overline{\QQ}$ induces an additive map
  $$
\varphi_{B,_*}: H_1(E(\CC);\ZZ) \to H_1(E(\CC);\ZZ)
$$
preserving the intersection product, and a $\overline{\QQ}$-linear map
$$
\varphi^*_{\dR} : H^1_{\dR}(E) \to H^1_{\dR}(E)
$$
preserving the subspace of differentials of first kind $H^0(E,\Omega^1_{E/\overline{\QQ}})$ and the de Rham cup product. Show that this induces an identity of the form
$$
\left(\begin{array}{cc}
        a & b \\
        c & d
\end{array}\right) \left(\begin{array}{cc}
        \omega_1 & \eta_1 \\
        \omega_2 & \eta_2
\end{array}\right) =  \left(\begin{array}{cc}
        \omega_1 & \eta_1 \\
        \omega_2 & \eta_2
\end{array}\right)\left(\begin{array}{cc}
        r & s \\
        0 & r^{-1}
\end{array}\right)
$$
with $a,b,c,d \in \ZZ$, $ad-bc=1$, and $r,s \in \overline{\QQ}$. Conclude that if $E$ has complex multiplication, then $\tau$ is quadratic imaginary, $\overline{\QQ}(\omega_1,\omega_2,\eta_1,\eta_2) = \overline{\QQ}(\omega_1,\eta_1)$, and $\overline{\QQ}(\omega_1,\eta_1)$ is algebraic over $\overline{\QQ}(2\pi i, \omega_1)$. 
\end{exo}

Loosely speaking, Grothendieck conjectured that algebraic cycles in powers of some algebraic variety $X$ are the \emph{only way} of producing algebraic relations between periods of $X$ (see \cite{andre04} 23.4.1 for a more precise statement). For elliptic curves, Schneider proved that each one of $\omega_1$, $\omega_2$, $\eta_1$ and $\eta_2$ are transcendental numbers, and Chudnovsky proved the uniform bound:
$$
\trdeg_{\QQ}\QQ(\omega_1,\omega_2,\eta_1,\eta_2)\ge 2
$$
for any elliptic curve, therefore establishing Grothendieck's period conjecture for complex multiplication elliptic curves (see \cite{waldschmidt06} and references therein). 

\begin{prop}
  With notation as above, if $\tau = \omega_2/\omega_1$, then $\Im (\tau)>0$ and
  $$
E_2(\tau) = 12 \left(\frac{\omega_1}{2\pi i}\right)\left(\frac{\eta_1}{2\pi i}\right)\text{, }E_4(\tau) = 12u \left(\frac{\omega_1}{2\pi i} \right)^{4}\text{, }E_6(\tau) = -216v\left(\frac{\omega_1}{2\pi i} \right)^6\text{.}
$$
\end{prop}

\begin{proof}
  The proof is based on the classical theory of elliptic and modular functions. Let $\Lambda = \ZZ\omega_1 + \ZZ\omega_2\subset \CC$. Then $\Lambda$ is a lattice and it follows from Weierstrass' uniformisation theorem that $u = g_2(\Lambda)$ and $v = g_3(\Lambda)$. On the other hand, if $\Lambda_{\tau} \defeq \omega_1^{-1}\Lambda = \ZZ + \ZZ \tau$, then
  $$
g_2(\Lambda_\tau) = \frac{(2\pi i)^4}{12}E_4(\tau)\text{, }g_3(\Lambda_\tau) =- \frac{(2\pi i)^6}{216} E_6(\tau)
$$
By homogeneity, we get
$$
g_2(\Lambda_{\tau}) = g_2(\omega_1^{-1}\Lambda) = \omega_1^4g_2(\Lambda) = \omega_1^4u 
$$
and similarly for $g_3$. This proves that
$$
E_4(\tau) = 12u \left(\frac{\omega_1}{2\pi i} \right)^{4}\text{, }E_6(\tau) = -216v\left(\frac{\omega_1}{2\pi i} \right)^6
$$
For $E_2(\tau)$, we use the formula (see \cite{serre73} eq. (46) p. 96)
$$
E_2(\tau) = -\frac{12}{(2\pi i)^2}\sum_{n}{\sum_m}' \frac{1}{(m+ n \tau)^2}
$$
where $\sum'$ means that $(m,n)\neq (0,0)$\footnote{Beware that the above sum does not converge absolutely, so that the order of the summation is important!}. Now, by using the following expression for the Weierstrass zeta function $\zeta_{\Lambda_{\tau}}$ (a primitive of $-\wp_{\Lambda_\tau}$)
$$
\zeta_{\Lambda_{\tau}}(z) = \frac{1}{z} =  \sum_{n}{\sum_m}'\left(\frac{1}{z-m - n\tau} + \frac{1}{m+ n\tau} + \frac{z}{(m + n\tau)^2} \right)\text{.}
$$
we obtain
$$
 \sum_{n}{\sum_m}' \frac{1}{(m+ n \tau)^2} = - \int_0^1 \wp_{\Lambda_{\tau}}(z)dz = -\omega_1 \eta_1\text{.}
$$
where the last equality follows from the identity
$$
\omega_1^{-2}\wp_{\Lambda_\tau}(z) = \wp_{\Lambda}(\omega_1z)\text{.} 
$$
We conclude that
$$
E_2(\tau) = 12 \left(\frac{\omega_1}{2\pi i}\right)\left(\frac{\eta_1}{2\pi i}\right)\text{.}
$$
\end{proof}

It follows from the above formulas and from Nesterenko's theorem that, for every elliptic curve over $\overline{\QQ}$,
$$
\trdeg_{\QQ}\QQ\left(e^{2\pi i \frac{\omega_2}{\omega_1}}, \frac{\omega_1}{2\pi i},\frac{\eta_1}{2\pi i}\right) = 3
$$
this improves the theorem of Chudnovsky, and thus also proves Grothendieck's Period conjecture for complex multiplication elliptic curves (cf. Exercise \ref{exo:cm}). 

\subsection{Open problems}\label{sec:problems}

Currently, the study of periods is an active and rapidly developing domain of number theory. Among the important recent achievements in the field there is Brown's theorem on multiple zeta values (see \cite{FB} for a thorough introduction).  

There remains countless open questions concerning periods. Besides Kontsevich's and Zagier's introduction \cite{KZ01}, Waldschmidt's survey \cite{waldschmidt06} contains a good summary of what is known and what is not regarding transcendence. Here, we focus only on values of (quasi)modular forms.

We have seen that Nesterenko's theorem proves, in particular, Grothendieck's period conjecture for complex multiplication elliptic curves. A careful analysis of the period conjecture for certain 1-motives attached to elliptic curves (see \cite{bertolin02}) suggests the following statement.

\begin{conj}\label{conj:modularforms}
  For any $\tau \in \mathbf{H}$, we have
  $$
  \trdeg_{\QQ}\QQ(2\pi i, \tau, e^{2\pi i \tau}, E_2(\tau), E_4(\tau), E_6(\tau)) \ge \begin{cases}
              3 & \text{if }\tau \text{ is quadratic imaginary}\\
              5 & \text{otherwise}                                                                        \end{cases}
            $$
            Moreover, we have equality if $j(\tau) \in \overline{\mathbf{Q}}$. 
\end{conj}

This conjectural statement essentially contains everything that is known or that there is to know concerning algebraic independence of values taken by quasimodular forms at a same $\tau \in \mathbf{H}$. For instance, both Nesterenko's theorem and Schneider's theorem on the $j$-function follow from the above conjecture.

\begin{exo}
Check that Conjecture \ref{conj:modularforms} implies Nesterenko's theorem and Schneider's theorem.
\end{exo}

\begin{obs}
A natural question is to ask what happens for elliptic modular forms of higher levels. From the point of view of transcendence, we don't get anything new. The values of higher level quasimodular forms (or modular functions) are algebraic on values of level 1 quasimodular forms. 
\end{obs}

A proof of Conjecture \ref{conj:modularforms} seems out of reach, but we can also turn our attention to higher dimensional notions of modular forms, such as Siegel or Hilbert modular forms. Understanding their values amounts to understanding periods of \emph{abelian varieties}, a generalisation of elliptic curves. This also includes periods of higher genus curves, since the cohomology in degree 1 (de Rham or Betti) of a curve is canonically isomorphic to the cohomology of its Jacobian.

We next exhibit some explicit examples of higher genera abelian periods.

\begin{ex}
  Let us consider the hyperelliptic curve $C$ over $\overline{\QQ}$ whose affine part is given by the equation $y^2=1-x^5$. The chart at $\infty$ is given by the equation $s^2=t^6-t$, where $(x,y)=(1/t,s/t^3)$. We now show how to compute the periods of $C$.

 For $k=1,2,3,4$, define the differential forms
  $$
\omega_k \defeq x^{k-1}\frac{dx}{y}\text{.}
$$
One can check that $\omega_1$ and $\omega_2$ are of first kind (i.e., everywhere regular), whereas $\omega_3$ and $\omega_4$ are of the second kind (i.e., all the residues vanish). Each of these forms define an element of $H^1_{\dR}(C)$, and since they have distinct orders at $\infty$, they must be linearly independent. As $\dim H^1_{\dR}(C) = 2\times \text{genus}(C) = 4$, they form a basis of $H^1_{\dR}(C)$.

Now, consider the path
$$
\epsilon : [0,1] \to C(\CC)\text{, }\qquad u\mapsto (u,\sqrt{1-u^5})\text{.}
$$
Using the automorphisms $\tau: (x,y)\mapsto (x,-y)$ and $\sigma:(x,y)\mapsto (\zeta x,y)$ of $C$, where $\zeta$ denotes a primitive 5th root of unity, we may define a loop $\gamma$ at $p=(0,1) \in C(\CC)$ by
$$
c\defeq \epsilon \cdot (\tau \circ \epsilon)^{-1} \cdot (\sigma \circ \tau \circ \epsilon) \cdot (\sigma \circ \epsilon)^{-1}\text{,}
$$
where $\cdot$ denotes path composition and ${}^{-1}$ the operation on paths that reverses direction. We compute (cf. formula (\ref{eq:betafunc})):
$$
\int_{\epsilon}\omega_k = \frac{1}{5}B\left(\frac{k}{5},\frac{1}{2}\right)\text{.}
$$
As $\tau^*\omega_k = -\omega_k$ and $\sigma^*\omega_k=\zeta^k\omega_k$, we conclude that
$$
\int_{\gamma}\omega_k = \int_{\epsilon}\omega_k -\int_{\epsilon}\tau^*\omega_k + \int_{\epsilon}(\sigma \circ \tau)^*\omega_k - \int_{\epsilon}\sigma^*\omega_k = \frac{2}{5}(1-\zeta^5)B\left(\frac{k}{5},\frac{1}{2}\right)\text{.}
$$
\end{ex}

\begin{exo}
  For $l=1,2,3,4$, let  $\gamma_l \defeq \sigma^l_*\gamma \in H_1(C(\CC),\QQ)=H^1_B(C)^{\vee}$. Show that $(\gamma_1,\ldots,\gamma_4)$ forms a basis of $H_1(C(\CC),\QQ)$, and that
  $$
\int_{\gamma_l}\omega_k = \frac{2}{5}\zeta^{k(l-1)}(1-\zeta^k)B\left(\frac{k}{5},\frac{1}{2}\right)\text{.}
$$
Note that these can be expressed in terms of the $\Gamma$ function via Euler's formula (\ref{eq:betafunc}).
\end{exo}

Here already not much is known. For instance, the period conjecture applied to the above example predicts that $\pi$, $\Gamma(1/5)$, and $\Gamma(2/5)$ should be algebraically independent over $\QQ$, i.e.,
\begin{equation}\label{eq:gpc-genus2}
\trdeg_{\QQ}(\pi,\Gamma(1/5),\Gamma(2/5))\stackrel{?}{=} 3\text{.}
\end{equation}
Currently, only the weaker
$$
\trdeg_{\QQ}(\pi,\Gamma(1/5),\Gamma(2/5))\ge 2
$$
is proved (see \cite{vasilev} and \cite{grinspan}).

\begin{obs}
If we restrict our attention to \emph{linear relations}, instead of arbitrary algebraic relations, then a general result is known. A theorem of Wüstholz (based on his analytic subgroup theorem) shows that every $\overline{\QQ}$-linear relation between periods of an abelian variety must come from an endomorphism of the abelian variety. This has been recently generalised by Huber-Wüstholz \cite{HW-19} to 1-motives.
\end{obs}

We have seen that interpreting periods of elliptic curves as values of modular forms can be useful, via Nesterenko's theorem, to understanding their algebraic independence properties. We can ask if a similar approach can be generalised to higher dimensions. As it was already remarked above, values of Hilbert or Siegel modular forms can be expressed in terms of periods of abelian varieties; our question then boils down to asking if Nesterenko's methods can be generalised to such modular forms of several variables.

Given the prominent role played by the Ramanujan equations in Nesterenko's method, one is naturally lead to the study of the differential equations in these higher dimensional contexts. This is indeed the point of view adopted by Pellarin \cite{pellarin05} for Hilbert modular forms. The case of Siegel modular forms was studied by Zudilin \cite{zudilin00}, via explicit equations involving theta functions, and by Bertrand-Zudilin \cite{BZ03} via derivatives of modular functions.

Recently, a geometric approach to such problems has been proposed in \cite{fonseca18}.

Recall that the Ramanujan equations can be interpreted as a vector field on some moduli space of elliptic curves with additional structure, and that this admits a natural generalisation to moduli spaces of abelian varieties. In \cite{fonseca18}, we also develop a similar theory for the Hilbert moduli problem.

For instance, consider the real quadratic field $\QQ(\sqrt{5})$. By considering `principally polarised abelian surfaces with real multiplication by $\QQ(\sqrt{5})$', we obtain a smooth quasi-affine variety $T$ over $\QQ$ of dimension 6 endowed with commuting algebraic vector fields $v_1$, $v_2$ (generalising the Ramanujan vector fields), and a canonical analytic map
$$
\varphi : \mathbf{H}^2 \to T(\CC)
$$
satisfying the differential equations
$$
\theta_j\varphi = v_j\circ \varphi\qquad (j=1,2)
$$
where
$$
\theta_1= \frac{1}{2\pi i}\left(\sqrt{5}^{-1}\frac{\partial}{\partial \tau_1} - \sqrt{5}^{-1}\frac{\partial}{\partial \tau_1}\right)\ \ \text{ and }\ \ \theta_2 = \frac{1}{2\pi i}\left(\frac{1+ \sqrt{5}^{-1}}{2}\frac{\partial}{\partial \tau_1} + \frac{1- \sqrt{5}^{-1}}{2}\frac{\partial}{\partial \tau_2} \right)\text{.} 
$$
This differential equation satisfies many of the remarkable properties the usual Ramanujan equations satisfy. For example, $\varphi$ can be shown to have integral coefficients in an appropriate $q$-expansion. Moreover, every leaf of the holomorphic foliation on $T(\CC)$ defined by $v_1$ and $v_2$ are Zariski-dense in $B_{\CC}$ (cf. Corollary \ref{coro:every-leaf} above).

One can moreover relate the values of $\varphi$ with periods of abelian surfaces with real multiplication by $\QQ(\sqrt{5})$. This allows us to reformulate the period conjecture in terms of bounds on the transcendence degree of the fields generated by values of $\varphi$, in the same spirit of Nesterenko's theorem. As it is shown \cite{fonseca18}, a successful adaptation of Nesterenko's method to this higher dimensional setting would yield in particular the conjectural statement (\ref{eq:gpc-genus2}).  

\begin{obs}
Although we don't make explicit mention to modular forms, it can be proved that $\varphi$ is indeed related to modular forms and their derivatives as in the Bertrand-Zudilin approach. We refer to \cite{fonseca18} Section 15 for a precise statement.
\end{obs}

There are many technical difficulties in obtaining sufficiently strong algebraic independence statements for the above higher dimensional generalisations of $(E_2,E_4,E_6)$, the main obstacle being the presence of positive dimensional `special subvarieties' in moduli spaces of abelian varieties.

Let us also remark that Nesterenko's method relies on a very restrictive growth condition satisfied by the Eisenstein series, namely the polynomial growth of their Fourier coefficients (see Section \ref{subsec:sketch} above). It turns out that this condition can be replaced by a much more flexible, geometric, notion of growth based on Nevanlinna theory (see \cite{fonseca19}), which is suitable to generalisation.

Here, a curious problem emerges. It is shown in \cite{fonseca19} that to any collection of holomorphic functions $f_1,\ldots,f_n$ on the complex unity disk $D$ satisfying an algebraic differential equation (with the $D$-property), an integrality property, and a mild growth condition, Nesterenko's method applies to give
$$
\trdeg_{\QQ}\QQ(f_1(z),\ldots,f_n(z))\ge n-1
$$
for every $z \in D\setminus\{0\}$.

This comprises Nesterenko's result if we take $(f_1,f_2,f_3,f_4)=(q,E_2(q),E_4(q),E_6(q))$, but in principle it could have other applications. It turns out that no other essentially different example (i.e., not related to elliptic modular forms) is currently known. This is surprising, given the rather general shape of the hypotheses.

It could be, however, that quasimodular forms are essentially the only functions on the disk satisfying the above conditions. If one could prove this fact, this would yield a rather exotic characterisation of quasimodular forms, making no explicit reference to modularity.

\end{document}